\documentclass[12pt]{article}

\usepackage{amssymb,amsthm,graphicx,amsmath}

\newtheorem{ttt}{TTT}
\newtheorem{thm}{Theorem}[section]
\newtheorem{lemma}[thm]{Lemma}
\newtheorem{cor}[thm]{Corollary}
\newtheorem{prop}[thm]{Proposition}
\newtheorem{exercise}[ttt]{Exercise}

\newtheorem{question}[ttt]{Question}
\newtheorem*{mainthm}{Main Theorem}
\newtheorem*{mainprop}{Main Proposition}
\newtheorem*{thm*}{Theorem}

\theoremstyle{remark}
\newtheorem{remark}[thm]{Remark}
\newtheorem{definition}[thm]{Definition}
\newtheorem{notation}[thm]{Notation}
\newtheorem*{warning*}{Caution}
\newtheorem{example}[thm]{Example}

\def\R{{\mathbb R}}
\def\Z{{\mathbb Z}}
\def\f{{\mathbb F}}
\def\F{{\mathbb F}}
\def\P{{\mathbb P}}
\def\N{{\mathbb N}}
\def\fg{{fg}}
\def\fp{{fp}}
\def\gad{\textsf{GAD}}
\def\GAD{\gad}
\def\clg{\textsf{CLG}}
\def\CLG{\clg}
\def\jsj{\textsf{JSJ}}
\def\JSJ{\jsj}
\def\qh{\textsf{QH}}
\def\QH{\qh}
\def\mr{\textsf{MR}}

\def\morphism{f}
\def\A{\mathcal A}
\def\PA{\mathcal{PA}}
\def\T{\mathcal{T}}

\def\G{\mathcal{G}}
\def\K{\mathcal{K}}
\def\gen{\mathcal B}
\def\llangle{\langle\!\langle}
\def\rrangle{\rangle\!\rangle}

\def\res{\phi}
\def\Cayley{T_\f}
\def\homo{f}

\def\stabker{\underrightarrow{Ker}\ }
\def\auto{\alpha}
\def\blackbox{B}

\def\limitgroup{\Gamma}
\def\pml{{\mathcal P}{\mathcal M}{\mathcal L}}
\def\peri{P}
\def\cl{\overline{P}}

\def\leaf{\ell}

\begin{document}

\title{Notes on Sela's work: Limit groups and Makanin-Razborov diagrams}
\author{Mladen Bestvina${}^*$ \and  Mark Feighn\thanks{The authors
    gratefully acknowledge support of the National Science
    Foundation. This material is based upon work supported by the National Science Foundation under Grants Nos.~DMS-0502441 and DMS-0805440.}}
\date{May 16, 2009}
\maketitle

\begin{abstract}
This is the first in a planned series of papers giving an alternate approach
to Zlil Sela's work on the Tarski problems. The present paper is an
exposition of work of Kharlampovich-Myasnikov and Sela giving a
parametrization of $Hom(G,\F)$ where $G$ is a finitely generated group
and $\F$ is a non-abelian free group.
\end{abstract}

\setcounter{tocdepth}{1}
\tableofcontents
\section{The Main Theorem}
\subsection{Introduction}
This is the first of a planned series of papers giving an alternative
approach to Zlil Sela's work on the Tarski problems
\cite{zs:tarski,zs:tarski1,zs:tarski2,zs:tarski3,zs:tarski4,zs:tarski5.1,zs:tarski5.2,zs:tarski6,zs:tarski7,zs:tarski8}.
The present paper is an exposition of the following result of
Kharlampovich-Myasnikov \cite{km:limit1,km:limit2} and Sela
\cite{zs:tarski1}:
\begin{thm*}\label{t:zmain}
Let $G$ be a finitely generated non-free group. There is a finite
collection $\{q_i:G\to\limitgroup_i\}$ of proper epimorphisms of $G$ such
that, for any homomorphism $\homo$ from $G$ to a free group $F$, there
is $\auto\in Aut(G)$ such that $\homo\auto$ factors through some
$q_i$.
\end{thm*}

A more refined statement is given in the Main Theorem on
page~\pageref{t:main}.  Our approach, though similar to Sela's,
differs in several aspects: notably a different measure of complexity
and a more geometric proof which avoids the use of the full Rips
theory for finitely generated groups acting on $\R$-trees; see
Section~\ref{s:more geometric}. We attempted to include enough
background material to make the paper self-contained. See Paulin
\cite{fp:asterisque} and Champetier-Guirardel \cite{cg:limit} for
accounts of some of Sela's work on the Tarski problems.

The first version of these notes was circulated in 2003. In the
meantime Henry Wilton \cite{henry} made available solutions to the
exercises in the notes. We also thank Wilton for making numerous
comments that led to many improvements.

\begin{remark}
In the theorem above, since $G$ is finitely generated we may assume
that $F$ is also finitely generated. If $F$ is abelian, then any $f$
factors through the abelianization of $G$ mod its torsion subgroup and
we are in the situation of Example~\ref{e:abelian} below. Finally, if
$F_1$ and $F_2$ are finitely generated non-abelian free groups then
there is an injection $F_1\to F_2$. So, if $\{q_i\}$ is a set of
epimorphisms that satisfies the conclusion of the theorem for maps to
$F_2$, then $\{q_i\}$ also works for maps to $F_1$. Therefore,
throughout the paper we work with a fixed finitely generated
non-abelian free group $\f$.
\end{remark}

\begin{notation}
Finitely generated (finitely presented) is abbreviated \fg\
  (respectively \fp).
\end{notation}

The main goal of \cite{zs:tarski1} is to give an answer to the following:

\begin{question}
Let $G$ be an \fg\ group.  Describe the set of all homomorphisms from
$G$ to $\f$.
\end{question}

\begin{example}
When $G$ is a free group, we can identify $Hom(G,\f)$ with the
cartesian product $\f^n$ where $n=rank(G)$.
\end{example}

\begin{example}\label{e:abelian}
If $G=\Z^n$, let $\mu:\Z^n\to\Z$ be the projection to one of the
coordinates. If $h:\Z^n\to \f$ is a homomorphism, there is an
automorphism $\alpha:\Z^n\to\Z^n$ such that $h\alpha$ factors through
$\mu$. This provides an explicit (although not 1-1) parametrization of
$Hom(G,\F)$ by $Aut(\Z^n)\times Hom(\Z,\F)\cong GL_n(\Z)\times \F$.
\end{example}

\begin{example}\label{e:closed surface}
When $G$ is the fundamental group of a closed genus $g$ orientable
surface, let $\mu:G\to F_g$ denote the homomorphism to a free group of
rank $g$ induced by the (obvious) retraction of the surface to the
rank $g$ graph. It is a folk theorem\footnote{see Zieschang
\cite{hz:surfaces} and Stallings \cite{js:surface}} that for every
homomorphism $\homo:G\to \f$ there is an automorphism $\alpha:G\to G$
(induced by a homeomorphism of the surface) so that $\homo\alpha$
factors through $\mu$.  The theorem was generalized to the case when
$G$ is the fundamental group of a non-orientable closed surface by
Grigorchuk and Kurchanov \cite{gk:surface}. Interestingly, in this
generality the single map $\mu$ is replaced by a finite collection
$\{\mu_1,\cdots,\mu_k\}$ of maps from $G$ to a free group $F$. In
other words, for all $\homo\in Hom(G,\f)$ there is $\auto\in Aut(G)$
induced by a homeomorphism of the surface such that $\homo\auto$
factors through some $\mu_i$.
\end{example}

\subsection{Basic properties of limit groups}
Another goal is to understand the class of groups that naturally
appear in the answer to the above question, these are called limit
groups.

\begin{definition}
Let $G$ be an \fg\ group. A sequence $\{\homo_i\}$ in $Hom(G,\f)$ is
{\it stable} if, for all $g\in G$, the sequence $\{\homo_i(g)\}$ is
eventually always 1 or eventually never 1. The {\it stable kernel} of $\{\homo_i\}$,
denoted $\stabker\homo_i,$ is
$$\{g\in G\mid \homo_i(g)=1\mbox{ for almost all }i\}.$$ An \fg\ group
$\limitgroup$ is a {\it limit group} if there is an \fg\ group $G$ and
a stable sequence $\{\homo_i\}$ in $Hom(G,\f)$ so
that $\limitgroup\cong G/\stabker\homo_i.$
\end{definition}

\begin{remark}
One can view each $\homo_i$ as inducing an action of $G$ on the Cayley
graph of $\f$, and then can pass to a limiting $\R$-tree action (after
a subsequence). If the limiting tree is not a line, then
$\stabker\homo_i$ is precisely the kernel of this action and so
$\limitgroup$ acts faithfully. This explains the name.
\end{remark}

\begin{definition}
An \fg\ group $\limitgroup$ is {\it residually free} if for every
element $\gamma\in\limitgroup$ there is $\homo\in Hom(\limitgroup,\f)$ such
that $\homo(\gamma)\not= 1$. It is {\it $\omega$-residually free} if for
every finite subset $X\subset\limitgroup$ there is $\homo\in
Hom(\limitgroup,\f)$ such that $\homo|X$ is injective.
\end{definition}

\begin{exercise}\label{e:torsion free}
Residually free groups are torsion free.
\end{exercise}

\begin{exercise}
Free groups and free abelian groups are $\omega$-residually free.
\end{exercise}

\begin{exercise}
The fundamental group of $n\P^2$
for $n=1$, 2, or 3 is not $\omega$-residually free, see \cite{rl:3p}.
\end{exercise}

\begin{exercise}
Every $\omega$-residually free group is a limit group.
\end{exercise}

\begin{exercise}\label{e:subgroup}
An \fg\ subgroup of an $\omega$-residually free group is
$\omega$-residually free. 
\end{exercise}

\begin{exercise}\label{e:max abelian}
Every non-trivial abelian subgroup of an $\omega$-residually free group is
contained in a unique maximal abelian subgroup. For example, $F\times
\Z$ is not $\omega$-residually free for any non-abelian $F$. 
\end{exercise}

\begin{lemma}\label{l:homs}
Let $G_1\to G_2\to\cdots$ be an infinite sequence of epimorphisms
between \fg\ groups. Then the sequence $$Hom(G_1,\F)\leftarrow
Hom(G_2,\F)\leftarrow\cdots$$ eventually stabilizes (consists of
bijections).
\end{lemma}

\begin{proof}
Embed $\f$ as a subgroup of $SL_2(\R)$. That the corresponding sequence
of varieties $Hom(G_i,SL_2(\R))$ stabilizes follows from algebraic
geometry, and this proves the lemma.
\end{proof}

\begin{cor}\label{c:homs}
A sequence of epimorphisms between $(\omega-)$residually free groups
eventually stabilizes.\qed
\end{cor}

\begin{lemma}\label{l:limit group is orf}
Every limit group is $\omega$-residually free.
\end{lemma}

\begin{proof}
Let $\limitgroup$ be a limit group, and let $G$ and $\{\homo_i\}$ be
as in the definition.  Without loss, $G$ is \fp. Now consider the
sequence of quotients
$$G\to G_1\to G_2\to\cdots\to \limitgroup$$ obtained by adjoining one
relation at a time. If $\Gamma$ is \fp\ the sequence terminates, and
in general it is infinite. Let $G'=G_j$ be such that
$Hom(G',\f)=Hom(\limitgroup,\f)$. All but finitely many $\homo_i$
factor through $G'$ since each added relation is sent to 1 by almost
all $\homo_i$. It follows that these $\homo_i$ factor through $\Gamma$
and each non-trivial element of $\Gamma$ is sent to 1 by only finitely
many $\homo_i$. By definition, $\limitgroup$ is $\omega$-residually
free.
\end{proof}

The next two exercises will not be used in this paper but are included
for their independent interest.

\begin{exercise}\label{embedding}
Every $\omega$-residually free group $\limitgroup$ embeds into $PSL_2(\R)$, and
also into $SO(3)$.
\end{exercise}

\begin{exercise}\label{embedding2}
Let $\limitgroup$ be $\omega$-residually free.  For any finite
collection of nontrivial elements $g_1,\cdots,g_k\in\limitgroup$ there
is an embedding $\limitgroup\to PSL_2(\R)$ whose image has no
parabolic elements and so that $g_1,\cdots,g_k$ go to hyperbolic
elements.
\end{exercise}

\subsection{Modular groups and the statement of the main theorem}
Only certain automorphisms, called {\it modular automorphisms}, are
needed in the theorem on page~\pageref{t:zmain}. This section contains
a definition of these automorphisms.

\begin{definition}
Free products with amalgamations and $HNN$-decompositions of a group
$G$ give rise to {\it Dehn twist automorphisms of $G$.} Specifically,
if $G=A*_C B$ and if $z$ is in the centralizer $Z_B(C)$ of $C$ in $B$,
then the automorphism $\auto_z$ of $G$, called the {\it Dehn twist in
$z$}, is determined as follows.
\[\auto_z(g)=
\begin{cases}
g, &\text{if $g\in A$;}\\
zgz^{-1}, &\text{if $g\in B.$}
\end{cases}\]
If $C\subset A$, $\phi:C\to A$ is a monomorphism, $G=A*_C=\langle
A,t\mid tat^{-1}=\phi(a), a\in A\rangle$,\footnote{$t$ is called a
  {\it stable letter.}} and $z\in Z_A(C)$, then
$\auto_z$ is determined as follows.
\[\auto_z(g)=
\begin{cases}
g, &\text{if $g\in A$;}\\
gz, &\text{if $g=t$.}
\end{cases}\]
\end{definition}

\begin{definition}
A \gad\footnote{\textsf{G}eneralized \textsf{A}belian
\textsf{D}ecomposition} of a group $G$ is a finite graph of groups
decomposition\footnote{We will use the terms {\it graph of groups
decomposition} and {\it splitting} interchangeably. Without further
notice, splittings are always {\it minimal},
i.e.\ the associated $G$-tree has no proper invariant subtrees.} 
of
$G$ with abelian edge groups in which some of the vertices are
designated \qh\footnote{\textsf{Q}uadratically \textsf{H}anging} and
some others are designated {\it abelian}, and the following holds.
\begin{itemize}
\item
A $\QH$-vertex group is the fundamental group of a compact surface $S$ with
boundary and the boundary components correspond to the incident edge
groups (they are all infinite cyclic). Further, $S$ carries a
pseudoAnosov homeomorphism (so $S$ is a torus with 1 boundary
component or $\chi(S)\le -2$).
\item
An abelian vertex group $A$ is non-cyclic abelian. 
Denote by $\peri(A)$ the subgroup of $A$ generated by incident edge
groups. The {\it peripheral subgroup of $A$}, denoted $\cl(A)$, is the
subgroup of $A$ that dies under every homomorphism from $A$ to $\Z$
that kills $\peri(A)$, i.e.\
$$\cl(A)=\cap\{Ker(\homo)\mid \homo\in Hom(A,\Z), \peri(A)\subset
Ker(\homo)\}.$$
\end{itemize}
The non-abelian  non-\qh\ vertices are {\it rigid}.
\end{definition}

\begin{remark}
We allow the possibility that edge and vertex groups of \gad 's are not \fg.
\end{remark}

\begin{remark}\label{r:retraction}
If $\Delta$ is a \gad\ for a \fg\ group $G$, and if $A$ is an abelian
vertex group of $\Delta$, then there are epimorphisms $G\to
A/\peri(A)\to A/\cl(A)$. Hence, $A/\peri(A)$ and $A/\cl(A)$ are
\fg. Since $A/\cl(A)$ is also torsion free, $A/\cl(A)$ is free, and so
$A=A_0\oplus\cl(A)$ with $A_0\cong A/\cl(A)$ a retract of
$G$. Similarly, $A/\cl(A)$ is a direct summand of $A/\peri(A)$. A
summand complementary to $A/\cl(A)$ in $A/\peri(A)$ must be a torsion
group by the definition of $\cl(A)$. In particular, $\peri(A)$ has
finite index in $\cl(A)$. It also follows from the definition of
$\cl(A)$ that any automorphism leaving $\peri(A)$ invariant must leave
$\cl(A)$ invariant as well. It follows that if $A$ is torsion free,
then any automorphism of $A$ that is the identity when restricted to
$\peri(A)$ is also the identity when restricted to
$\cl(A)$.
\end{remark}

\begin{definition}\label{d:mod}
The {\it modular group} $Mod(\Delta)$ associated to a \gad\ $\Delta$
of $G$ is the subgroup of $Aut(G)$ generated by
\begin{itemize}
\item
inner automorphisms of $G$,
\item 
Dehn twists in
elements of $G$ that centralize an edge group of $\Delta$,
\item 
unimodular\footnote{The induced automorphism of $A/\cl(A)$ has
determinant 1.} automorphisms of an abelian vertex group that are
the identity on its peripheral subgroup and all other vertex groups, and
\item 
automorphisms induced by homeomorphisms of surfaces $S$ underlying
\QH-vertices that fix all boundary components. If $S$ is closed and
orientable, we require the homeomorphisms to be
orientation-preserving\footnote{We will want our homeomorphisms to be
products of Dehn twists.}.
\end{itemize}
The {\it modular group of $G$},
denoted $Mod(G)$, is the subgroup of $Aut(G)$ generated by
$Mod(\Delta)$ for all \gad's $\Delta$ of $G$.
At times it will be convenient to view $Mod(G)$ as a subgroup of
$Out(G)$. In particular, we will say that an element of $Mod(G)$ is
{\it trivial} if it is an inner automorphism.
\end{definition}

\begin{definition}\label{d:gdt}
A {\it generalized Dehn twist} is a Dehn twist or an automorphism
$\auto$ of $G=A*_C B$ or $G=A*_C$ where in each case $A$ is abelian,
$\auto$ restricted to $\cl(A)$ and $B$ is the identity, and $\auto$
induces a unimodular automorphism of $A/\cl(A)$. Here $\cl(A)$ is the
peripheral subgroup of $A$ when we view $A*_C B$ or $G=A$ as a \gad\
with one or zero edges and abelian vertex $A$. If $C$
is an edge groups of a \gad\ for $G$ and if $z\in Z_G(C)$, then $C$
determines a splitting of $G$ as above and so also a Dehn twist in
$z$. Similarly, an abelian vertex $A$ of a \gad\
determines\footnote{by folding together the edges incident to $A$} a
splitting $A*_C B$ and so also generalized Dehn
twists.
\end{definition}

\begin{exercise}\label{e:generators}
$Mod(G)$ is generated by inner automorphisms together with generalized
Dehn twists.
\end{exercise}

\begin{definition}
A {\it factor set} for a group $G$ is a finite collection of proper
epimorphisms $\{q_i:G\to G_i\}$ such that if $\homo\in Hom(G,\f)$ then
there is $\auto\in Mod(G)$ such that $\homo\auto$ factors through some
$q_i$.
\end{definition}

\begin{mainthm}[\cite{km:limit1,km:limit2,zs:tarski}]\label{t:main}
Let $G$ be an \fg\ group that is not free. Then, $G$ has a factor set
$\{q_i:G\to \limitgroup_i\}$ with each $\limitgroup_i$ a limit
group. If $G$ is not a limit group, we can always take $\alpha$ to be
the identity.
\end{mainthm}

We will give two proofs--one in Section~\ref{s:proof} and the second,
which uses less in the way of technical machinery, in
Section~\ref{s:more geometric}. In the remainder of this section, we
explore some consequences of the Main Theorem and then give another
description of limit groups.

\subsection{Makanin-Razborov diagrams}
\begin{cor}\label{c:main}
Iterating the construction of the Main Theorem (for $\limitgroup_i$'s etc.)
yields a finite tree of groups terminating in groups that are free.
\end{cor}

\begin{proof}
If $\limitgroup\to \limitgroup'$ is a proper epimorphism between limit
groups, then since limit groups are residually free,
$Hom(\limitgroup',\f)\subsetneq Hom(\limitgroup,\f)$. We are done by
Lemma~\ref{l:homs}.
\end{proof}

\begin{definition}\label{d:mr}
The tree of groups and epimorphisms provided by Corollary~\ref{c:main}
is called an {\it \mr-diagram}\footnote{for
\textsf{M}akanin-\textsf{R}azborov, cf.\
\cite{gm:equations,gm:equations2, ar:equations}.} for $G$ (with
respect to $\F$). If
$$G\overset{q}{\to}\limitgroup_1\overset{q_1}{\to}\limitgroup_2\overset{q_2}{\to}\cdots\overset{q_{m-1}}{\to}\limitgroup_m$$
is a branch of an \mr-diagram and if $\homo\in Hom(G,\f)$ then we say
that $\homo$ {\it \mr-factors} through this branch if there are
$\auto\in Mod(G)$ (which is the identity if $G$ is not a limit group),
$\auto_i\in Mod(\limitgroup_i)$, for $1\le i< m$, and $\homo'\in
Hom(\limitgroup_m,\F)$ (recall $\limitgroup_m$ is
free) such that $\homo=\homo'q_{m-1}\auto_{m-1}\cdots q_1\auto_1q\auto$.
\end{definition}

\begin{remark}\label{r:mr}
The key property of an \mr-diagram for $G$ is that, for $\homo\in
Hom(G,\f)$, there is a branch of the diagram through which $\homo$
\mr-factors. This provides an answer to Question~1 in that $Hom(G,\f)$
is parametrized by branches of an \mr-diagram and, for each branch as
above, $Mod(G)\times Mod(\limitgroup_1)\times\cdots\times
Mod(\limitgroup_{m-1})\times Hom(\limitgroup_m,\f)$. Note that if
$\limitgroup_m$ has rank $n$, then $Hom(\limitgroup_m,\F)\cong\F^n$.
\end{remark}

In \cite{zs:tarski8}, Sela constructed \mr-diagrams with respect to
hyperbolic groups. In her thesis \cite{ea:thesis}, Emina Alibegovi\'c
constructed \mr-diagrams with respect to limit groups. More recently,
Daniel Groves \cite{dg:mr1,dg:mr2} constructed \mr-diagrams with
respect to torsion-free groups that are hyperbolic relative to a
collection of free abelian subgroups.

\subsection{Abelian subgroups of limit groups}
\begin{cor}\label{c:abelian fg}
Abelian subgroups of limit groups are \fg\ and free.
\end{cor}

Along with the Main Theorem, the proof of Corollary~\ref{c:abelian
  fg} will depend on an exercise and two lemmas.

\begin{exercise}[{\cite[Lemma~2.3]{zs:tarski1}}]\label{e:restriction
  to M}
Let $M$ be a non-cyclic maximal abelian subgroup of the limit group
$\limitgroup$.
\begin{enumerate}
\item
If $\limitgroup=A*_C B$ with $C$ abelian, then $M$ is conjugate into
$A$ or $B$.
\item
If $\limitgroup=A*_C$ with $C$ abelian, then either $M$ is conjugate
into $A$ or there is a stable letter $t$ such that $M$ is conjugate to
$M'=\langle C,t\rangle$ and $\limitgroup=A*_C M'$.
\end{enumerate}
As a consequence, if $\auto\in Mod(\limitgroup)$ is a generalized Dehn
twist and $\auto|M$ is non-trivial, then there is an element
$\gamma\in\limitgroup$ and a \gad\ $\Delta=M*_C B$ or $\Delta=M$ for
$\limitgroup$ such that, up to conjugation by $\gamma$, $\auto$ is
induced by a unimodular automorphism of $M/\cl(M)$ (as in
Definition~\ref{d:gdt}). (Hint: Use Exercise~\ref{e:max abelian}.)
\end{exercise}

\begin{lemma}\label{l:inject}
Suppose that $\limitgroup$ is a limit group with factor set
$\{q_i:\limitgroup\to G_i\}$. If $H$ is a (not necessarily \fg)
subgroup of $\limitgroup$ such that, for every homomorphism $\homo:
\limitgroup\to\f$, $\homo|H$ factors through some $q_i|H$
(pre-compositions by automorphisms of $\limitgroup$ not needed) then,
for some $i$, $q_i|H$ is injective.
\end{lemma}

\begin{proof}
Suppose not and let $1\not= h_i\in Ker(q_i|H).$ Since $\limitgroup$ is
a limit group, there is $\homo\in Hom(\limitgroup,\f)$ that is
injective on $\{1,h_1,\cdots, h_n\}$. On the other hand, $\homo|H$
factors through some $q_i|H$ and so $h_i=1$, a contradiction.
\end{proof}

\begin{lemma}\label{l:epi}
Let $M$ be a non-cyclic maximal abelian subgroup of the limit group
$\limitgroup$. There is an epimorphism $r:\limitgroup\to A$ where $A$
is free abelian and every modular automorphism of\, $\limitgroup$ is
trivial\footnote{agrees with the restriction of an inner automorphism
of $\limitgroup$.}  when restricted to $M\cap Ker(r)$.
\end{lemma}

\begin{proof}
By Exercise~\ref{e:generators}, it is enough to find $r$ such that
$\auto|M\cap Ker(r)$ is trivial for every generalized Dehn twist
$\auto\in Mod(\limitgroup)$. By Exercise~\ref{e:restriction to M} and
Remark~\ref{r:retraction}, there is a \fg\ free abelian subgroup
$M_\auto$ of $M$ and a retraction $r_\auto:\limitgroup\to M_\auto$
such that $\auto|M\cap Ker(r_\auto)$ is trivial. Let
$r=\Pi_{\auto}r_\auto: \limitgroup\to\Pi_\auto M_\auto$ and let $A$ be
the image of $r$. Since $\limitgroup$ is \fg, so is $A$. Hence $A$ is
free abelian.
\end{proof}

\begin{proof}[Proof of Corollary~\ref{c:abelian fg}]
Let $M$ be a maximal abelian subgroup of a limit group
$\limitgroup$. We may assume that $M$ is not cyclic. Since
$\limitgroup$ is torsion free, it is enough to show that $M$ is
\fg. By restricting the map $r$ of Lemma~\ref{l:epi} to $M$, we see
that $M=A\oplus A'$ where $A$ is \fg\ and each $\auto|A'$ is trivial. Let
$\{q_i:\limitgroup\to\limitgroup_i\}$ be a factor set for
$\limitgroup$ given by Theorem~\ref{t:main}. By Lemma~\ref{l:inject},
$A'$ injects into some $\limitgroup_i$. Since
$Hom(\limitgroup_i,\f)\subsetneq Hom(\limitgroup,\f)$, we may conclude
by induction that $A'$ and hence $M$ is \fg.
\end{proof}

\subsection{Constructible limit groups}
It will turn out that limit groups can be built up inductively from
simpler limit groups. In this section, we give this description and
list some properties that follow.
\begin{definition}\label{d:clg}
We define a hierarchy of \fg\ groups -- if a group belongs to this
hierarchy it is called a \CLG\footnote{\textsf{C}onstructible
\textsf{L}imit \textsf{G}roup}.

Level 0 of the hierarchy consists of \fg\ free groups.

A group $\limitgroup$ belongs to level $\leq n+1$ iff either it has a
free product decomposition $\limitgroup=\limitgroup_1*\limitgroup_2$
with $\limitgroup_1$ and $\limitgroup_2$ of level $\leq n$ or it has a
homomorphism $\rho:\limitgroup\to\limitgroup'$ with $\limitgroup'$ of
level $\leq n$ and it has a \gad\ such that
\begin{itemize}
\item $\rho$ is injective on the peripheral subgroup of each abelian
vertex group.
\item $\rho$ is injective on each edge group $E$ and at least one of
  the images of $E$ in a vertex group of the
  one-edged splitting induced by $E$ is a maximal abelian subgroup.
\item The image of each $\QH$-vertex group is a non-abelian subgroup of
$\limitgroup'$.
\item For every rigid vertex group $\blackbox$, $\rho$
is injective on the {\it envelope $\tilde\blackbox$ of $\blackbox$},
defined by first replacing each abelian vertex with the peripheral
subgroup and then letting $\tilde\blackbox$ be the subgroup of the
resulting group generated by $\blackbox$ and by the centralizers of
incident edge-groups.
\end{itemize}
\end{definition}

\begin{example}
A \fg\ free abelian group is a \clg\ of level one (consider a one-point
\gad\ for $\Z^n$ and $\rho:\Z^n\to\langle
0\rangle$). The fundamental group of a closed surface $S$ with
$\chi(S)\le -2$ is a \clg\ of level one. For example, an orientable
genus 2 surface is a union of 2 punctured tori and the retraction to
one of them determines $\rho$. Similarly, a non-orientable genus 2
surface is the union of 2 punctured Klein bottles.
\end{example}

\begin{example}
Start with the circle and attach to it 3 surfaces with one boundary
component, with genera 1, 2, and 3 say. There is a retraction to the
surface of genus 3 that is the union of the attached surfaces of genus
1 and 2. This retraction sends the genus 3 attached surface say to the
genus 2 attached surface by ``pinching a handle''. The \gad\ has a
central vertex labeled $\Z$ and there are 3 edges that emanate from
it, also labeled $\Z$. Their other endpoints are $\QH$-vertex
groups. The map induced by retraction satisfies the requirements so
the fundamental group of the 2-complex built is a \CLG.
\end{example}

\begin{example}
Choose a primitive\footnote{no proper root} $w$ in the \fg\ free group
$F$ and form $\limitgroup=F*_\Z F$, the {\it double of $F$ along $\langle
w\rangle$} (so $1\in \Z$ is identified with $w$ on both sides). There
is a retraction $\limitgroup\to F$ that satisfies the requirements
(both vertices are rigid), so $\limitgroup$ is a \CLG.
\end{example}

The following can be proved by induction on levels.

\begin{exercise}
Every \CLG\ is \fp , in fact coherent. Every \fg\ subgroup of a \CLG\ is
a \CLG. (Hint: a graph of coherent groups over \fg\ abelian
  groups is coherent.)
\end{exercise}

\begin{exercise}
Every abelian subgroup of a \CLG\ $\limitgroup$ is \fg\ and free, and
there is a uniform bound to the rank. There is a finite $K(\limitgroup,1)$.
\end{exercise}

\begin{exercise}\label{e:principal splitting}
Every non-abelian, freely indecomposable \CLG\ admits a
{\normalfont principal splitting} over $\Z$: $A*_\Z B$ or $A*_\Z$ with $A$, $B$
non-cyclic, and in the latter case $\Z$ is maximal abelian in the whole
group. 
\end{exercise}

\begin{exercise}\label{e:clg}
Every \CLG\ is $\omega$-residually free.
\end{exercise}

The last exercise is more difficult than the others. It explains where
the conditions in the definition of \CLG\ come from. The idea is to
construct homomorphisms $G\to \F$ by choosing complicated modular
automorphisms of $G$, composing with $\rho$ and then with a
homomorphism to $\F$ that comes from the inductive assumption.

\begin{example}
Consider an index 2 subgroup $H$ of an \fg\ free group $F$ and choose
$g\in F\setminus H$. Suppose that $G:=H*_{\langle g^2\rangle}\langle
g\rangle$ is freely indecomposable and admits no principal cyclic
splitting. There is the obvious map $G\to F$, but
$G$ is not a limit group (Exercise~\ref{e:principal splitting} and
Theorem~\ref{t:tfae}). This shows the necessity of the last condition
in the definition of \clg's.
\footnote{The element $g:=a^2b^2a^{-2}b^{-1}\notin H:=\langle a, b^2,
bab^{-1} \rangle\subset F:=\langle a,b\rangle$ is such an
example. This can be seen from the fact that if $\langle x,y,z\rangle$
denotes the displayed basis for $H$, then
$g^2=x^2yx^{-2}y^{-1}z^2yz^{-2}$ is Whitehead reduced and each basis
element occurs at least 3 times.}
\end{example}

In Section~\ref{s:tfae}, we will show:
\begin{thm}\label{t:tfae}
For an \fg\ group $G$, the following are equivalent.
\begin{enumerate}
\item $G$ is a \CLG.
\item $G$ is $\omega$-residually free.
\item $G$ is a limit group.
\end{enumerate}
\end{thm}

The fact that $\omega$-residually free groups are \clg's is due to
O.~Kharlampovich and A.~Myasnikov \cite{km:limitgroups}.  Limit groups
act freely on $\R^n$-trees; see Remeslennikov \cite{vr:rntrees} and
Guirardel \cite{vg:limitgroups}.  Kharlampovich-Myasnikov
\cite{km:limit2} prove that limit groups act freely on $\Z^n$-trees
where $\Z^n$ is lexicographically ordered. Remeslennikov
\cite{vr:2resfree} also demonstrated that 2-residually free groups are
$\omega$-residually free.

\section{The Main Proposition}
\begin{definition}
An \fg\ group is {\it generic} if it is torsion free, freely
indecomposable, non-abelian, and not a closed surface group.
\end{definition}

The Main Theorem will follow from the next proposition.
\begin{mainprop}
Generic limit groups have factor sets.
\end{mainprop}

Before proving this proposition, we show how it implies the Main
Theorem.
\begin{definition}
Let $G$ and $G'$ be \fg\ groups. The minimal number of generators for
$G$ is denoted $\mu(G)$. We say that $G$ is {\it simpler} than $G'$ if
there is an epimorphism $G'\to G$ and either $\mu(G)<\mu(G')$ or
$\mu(G)=\mu(G')$ and $Hom(G,\f)\subsetneq Hom(G',\f)$.
\end{definition}

\begin{remark}\label{r:complexity}
It follows from Lemma~\ref{l:homs} that every sequence $\{G_i\}$ with
$G_{i+1}$ simpler than $G_i$ is finite.
\end{remark}

\begin{definition}
If $G$ is an \fg\ group, then by $RF(G)$ denote the {\it universal
residually free quotient of $G$}, i.e.\ the quotient of $G$ by the
(normal) subgroup consisting of elements killed by every homomorphism
$G\to\f$.
\end{definition}

\begin{remark}\label{r:rf} 
$Hom(G,\f)=Hom(RF(G),\f)$ and for every proper quotient $G'$ of $RF(G)$,
$Hom(G',\f)\subsetneq Hom(G,\f)$.
\end{remark}

\begin{proof}[The Main Proposition implies the Main Theorem]
Suppose that $G$ is an \fg\ group that is not free. By
Remark~\ref{r:complexity}, we may assume that the Main Theorem holds
for groups that are simpler than $G$. By Remark~\ref{r:rf}, we may
assume that $G$ is residually free, and so also torsion
free. Examples~\ref{e:abelian} and \ref{e:closed surface} show that
the Main Theorem is true for abelian and closed surface groups. If
$G=U*V$ with $U$ non-free and freely indecomposable and with $V$
non-trivial, then $U$ is simpler than $G$. So, $U$ has a factor set
$\{q_i:U\to L_i\}$, and $\{q_i*Id_V:U*V\to L_i*V\}$ is a factor
set for $G$.

If $G$ is not a limit group, then there is a non-empty finite subset
$\{g_i\}$ of $G$ such that any homomorphism $G\to\f$ kills
one of the $g_i$. We then have a factor set $\{G\to H_i:=G/\llangle
g_i\rrangle\}$. Since $Hom(H_i,\f)\subsetneq Hom(G,\f)$, by induction the
Main Theorem holds for $H_i$ and so for $G$.

If $G$ is generic and a limit group, then the Main Proposition gives a
factor set $\{q_i:G\to G_i\}$ for $G$. Since $G$ is residually free,
each $G_i$ is simpler than $G$. We are assuming that the Main Theorem
then holds for each $G_i$ and this implies the result for $G$.
\end{proof}

\section{Review: Measured laminations and $\R$-trees}\label{s:jsj}
The proof of the Main Proposition will use a theorem of Sela
describing the structure of certain real trees. This in turn depends
on the structure of measured laminations. In Section~\ref{s:more
geometric}, we will give an alternate approach that only uses the
lamination results. First these concepts are reviewed. A more
leisurely review with references is \cite{mb:handbook}.

\subsection{Laminations}
\begin{definition}
A {\it measured lamination} $\Lambda$ on a simplicial 2-complex $K$
consists of a closed subset $|\Lambda|\subset |K|$ and a {\it
transverse measure} $\mu$. $|\Lambda|$ is disjoint from the vertex
set, intersects each edge in a Cantor set or empty set, and intersects
each 2-simplex in 0, 1, 2, or 3 families of straight line segments
spanning distinct sides. The measure $\mu$ assigns a non-negative
number $\int_I \mu$ to every interval $I$ in an edge whose endpoints
are outside $|\Lambda|$. There are two conditions:
\begin{enumerate}
\item
{\bf (compatibility)} If two intervals $I$, $J$ in two sides of the
same triangle $\Delta$ intersect the same components of
$|\Lambda|\cap\Delta$ then $\int_I \mu=\int_J \mu$.
\item
{\bf (regularity)} $\mu$ restricted to an edge is equivalent under a
``Cantor function'' to the Lebesgue measure on an interval in $\R$.
\end{enumerate}

A path component of $|\Lambda|$ is a {\it leaf}.
\end{definition}

Two measured laminations on $K$ are considered equivalent if they
assign the same value to each edge.

\begin{prop}[Morgan-Shalen \cite{ms:valuations1}]
Let $\Lambda$ be a measured lamination on compact $K$. Then
$$\Lambda=\Lambda_1\sqcup\cdots\sqcup\Lambda_k$$ so that each
$\Lambda_i$ is either {\normalfont minimal} (each leaf is dense in
$|\Lambda_i|$) or {\normalfont simplicial} (each leaf is compact, a
regular neighborhood of $|\Lambda_i|$ is an $I$-bundle over a leaf and
$|\Lambda_i|$ is a Cantor set subbundle).
\end{prop} 

There is a theory, called the {\it Rips machine}, for analyzing minimal
measured laminations. It turns out that there are only 3 qualities.

\begin{example}[Surface type]
Let $S$ be a compact hyperbolic surface (possibly with totally
geodesic boundary). If $S$ admits a pseudoAnosov homeomorphism then it
also admits {\it filling measured geodesic laminations} -- these are
measured laminations $\Lambda$ (with respect to an appropriate
triangulation) such that each leaf is a biinfinite geodesic and all
complementary components are ideal polygons or crowns. Now to get the
model for a general surface type lamination attach finitely many
annuli $S^1\times I$ with lamination $S^1\times \mbox{(Cantor set)}$
to the surface along arcs transverse to the geodesic lamination. If
these additional annuli do not appear then the lamination is of
{\it pure surface type}. See Figure~\ref{f:surface}.
\end{example}

\begin{figure}
\includegraphics{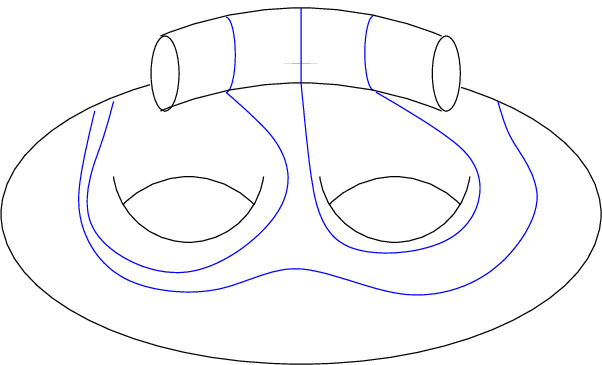}
\caption{A surface with an additional annulus and some pieces of
leaves.}\label{f:surface}
\end{figure}

\begin{example}[Toral type]
Fix a closed interval
$I\subset \R$, a finite collection of pairs $(J_i,J'_i)$ of closed
intervals in $I$, and isometries $\gamma_i:J_i\to J_i'$ so that:

\begin{enumerate}
\item
If $\gamma_i$ is orientation reversing then $J_i=J_i'$ and the
midpoint is fixed by $\gamma_i$.
\item The length of the intersection of all $J_i$ and $J_i'$ (over all
$i$) is more than twice the translation length of each orientation
preserving $\gamma_i$ and the fixed points of all orientation
reversing $\gamma_i$ are in the middle third of the intersection.
\end{enumerate} 

Now glue a foliated band for each pair $(J_i,J_i')$ so that following
the band maps $J_i$ to $J_i'$ via $\gamma_i$. Finally, using Cantor
functions blow up the foliation to a lamination. There is no need to
explicitly allow adding annuli as in the surface case since they
correspond to $\gamma_i=Id$. The subgroup of $Isom(\R)$ generated
by the extensions of the $\gamma_i$'s is the {\it Bass group}. The
lamination is minimal iff its Bass group is not discrete.
\end{example}

\begin{example}[Thin type]
This is the most mysterious type of all. It was discovered by Gilbert
Levitt, see \cite{gl:thin}. In the {\it pure} case (no annuli
attached) the leaves are 1-ended trees (so this type naturally lives
on a 2-complex, not on a manifold). By performing certain moves
(sliding, collapsing) that don't change the homotopy type (respecting
the lamination) of the complex one can transform it to one that
contains a (thin) band. This band induces a non-trivial free product
decomposition of $\pi_1(K)$, assuming that the component is a part of
a resolution of a tree (what's needed is that loops that follow leaves
until they come close to the starting point and then they close up are
non-trivial in $\pi_1$).

In the general case we allow additional annuli to be glued, just like
in the surface case. Leaves are then 1-ended trees with circles
attached.
\end{example}

\begin{thm}[``Rips machine'']\label{t:rips}
Let $\Lambda$ be a measured lamination on a finite 2-complex $K$, and
let $\Lambda_i$ be a minimal component of $\Lambda$. There is a
neighborhood $N$ (we refer to it as a {\normalfont standard}
neighborhood) of $|\Lambda_i|$, a finite 2-complex $N'$ with measured
lamination $\Lambda'$ as in one of 3 model examples, and there is a
$\pi_1$-isomorphism $f:N\to N'$ such that $f^*(\Lambda')=\Lambda$.
\end{thm}

We refer to $\Lambda_i$ as being of {\it surface}, {\it toral}, or
{\it thin} type.

\subsection{Dual trees}
Let $G$ be an \fg\ group and let $\hat K$ be a simply connected
2-dimensional simplicial $G$-complex so that, for each simplex
$\Delta$ of $\hat K$,
$Stab(\Delta)=Fix(\Delta)$.\footnote{$Stab(\Delta):=\{g\in G\mid
g\Delta=\Delta\}$ and $Fix(\Delta):=\{g\in G\mid gx=x, x\in\Delta\}$}
Let $\hat\Lambda$ be a $G$-invariant measured lamination in $\hat
K$. There is an associated real $G$-tree $T(\hat\Lambda)$ constructed as
follows.  Consider the pseudo-metric on $\hat K$ obtained by
minimizing the $\hat\Lambda$-length of paths between points. The real
tree $T(\hat\Lambda)$ is the associated metric space\footnote{identify
points of pseudo-distance 0}. There is a natural map $\hat K\to
T(\hat\Lambda)$ and we say that $(\hat K, \hat\Lambda)$ is a {\it model}
for $T(\hat\Lambda)$ if
\begin{itemize}
\item
for each edge $\hat e$ of $\hat K$, $T(\hat\Lambda\mid\hat e)\to
T(\hat\Lambda)$ is an isometry (onto its image) and
\item
the quotient $\hat K/G$ is compact. 
\end{itemize}
If a tree $T$ admits a model $(\hat K,\hat\Lambda)$, then we say that
$T$ is {\it dual} to $(\hat K,\hat\Lambda)$. This is denoted
$T=Dual(\hat K,\hat\Lambda)$. We will use the quotient
$(K,\Lambda):=(\hat K,\hat\Lambda)/G$ with simplices decorated (or
labeled) with stabilizers to present a model and sometimes abuse
notation by calling $(K,\Lambda)$ a model for $T$.

\begin{remark}
Often the $G$-action on $\hat K$ is required to be free. We have
relaxed this condition in order to be able to consider actions of \fg\
groups. For example, if $T$ is a minimal\footnote{no proper invariant
subtrees}, simplicial $G$-tree (with the metric where edges have
length one\footnote{This is called the {\it simplicial metric} on
$T$.}) then there is a lamination $\hat\Lambda$ in $T$ such that
$Dual(T,\hat\Lambda)=T$.\footnote{The metric and simplicial topologies
on $T$ don't agree unless $T$ is locally finite. But, the action of
$G$ is by isomorphisms in each structure. So, we will be
sloppy and ignore this distinction.}
\end{remark}

If $S$ and $T$ are real $G$-trees, then an equivariant map
$\morphism:S\to T$ is a {\it morphism} if every compact segment of $S$
has a finite partition such that the restriction of $\morphism$ to
each element is an isometry or trivial\footnote{has image a point}.

If $S$ is a real $G$-tree with $G$ \fp, then there is a real $G$-tree
$T$ with a model and a morphism $\morphism: T\to S$. The map
$\morphism$ is obtained by constructing an equivariant map to $S$ from
the universal cover of a 2-complex with fundamental group $G$. In
general, if $(\hat K,\hat\Lambda)$ is a model for $T$ and if $T\to S$
is a morphism then the composition $\hat K\to T\to S$ is a {\it
resolution} of $S$.

\subsection{The structure theorem}
Here we discuss a structure theorem (see Theorem~\ref{t:structure}) of
Sela for certain actions of an \fg\ torsion free group $G$ on real
trees. The actions we consider will usually be super
stable\footnote{If $J\subset I$ are (non-de\-gen\-er\-ate) arcs in $T$
and if $Fix_T(I)$ is non-trivial, then $Fix_T(J)=Fix_T(I)$}, have
primitive\footnote{root-closed} abelian (non-degenerate) arc
stabilizers, and have trivial tripod\footnote{a cone on 3 points}
stabilizers. There is a short list of basic examples.

\begin{example}[Pure surface type]
A real $G$-tree $T$ is {\it of pure surface type} if it is dual to the
universal cover of $(K,\Lambda)$ where $K$ is a compact surface and
$\Lambda$ is of pure surface type.  We will usually use the alternate
model where boundary components are crushed to points and are labeled
$\mathbb Z$.
\end{example}

\begin{example}[Linear] 
The tree $T$ is {\it linear} if $G$ is abelian, $T$ is a line and
there an epimorphism $G\to\mathbb Z^n$ such that $G$ acts on $T$ via a
free $\mathbb Z^n$-action on $T$. In particular, $T$ is dual to $(\hat
K,\hat\Lambda)$ where $\hat K$ is the universal cover of the
2-skeleton of an $n$-torus $K$. For simplicity, we often complete $K$
with its lamination to the whole torus. This is a special case of a
toral lamination.
\end{example}

\begin{example}[Pure thin]
The tree $T$ is {\it pure thin} if it is dual to the universal cover
of a finite 2-complex $K$ with a pure thin lamination $\Lambda$. If
$T$ is pure thin then $G\cong F*V_1*\cdots *V_m$ where $F$ is
non-trivial and \fg\ free and $\{V_1,\cdots,V_m\}$ represents the
conjugacy classes of non-trivial point stabilizers in $T$.
\end{example}

\begin{example}[Simplicial]
The tree $T$ is {\it simplicial} if it is dual to $(\hat
K,\hat\Lambda)$ where all leaves of $\Lambda:=\hat\Lambda/G$ are
compact. If $T$ is simplicial it is convenient to crush the leaves and
complementary components to points in which case $\hat K$ becomes a
tree isomorphic to $T$.
\end{example}

If $\K$ is a graph of 2-complexes with underlying graph of groups
$\G$\footnote{for each bonding map $\phi_e:G_e\to G_v$ there are
simplicial $G_e$- and $G_v$-complexes $\hat K_e$ and $\hat K_v$
together with a $\phi_e$-equivariant simplicial map $\Phi_e:\hat
K_e\to \hat K_v$} then there is a simplicial $\pi_1(\G)$-space $\hat
K(\K)$ obtained by gluing copies of $\hat K_e\times I$ and $\hat K_v$'s
equipped with a simplicial $\pi_1(\G)$-map $\hat K(\K)\to T(\G)$ that
crushes to points copies of $\hat K_e\times\{point\}$ as well as the
$\hat K_v$'s.

\begin{definition}
A real $G$-tree is {\it very small} if the action is
non-trivial\footnote{no point is fixed by $G$}, minimal, the
stabilizers of non-degenerate arcs are primitive abelian, and the
stabilizers of non-degenerate tripods are trivial.
\end{definition}

\begin{thm}[{\cite[special case of Section~3]{zs:accessibility}} See also \cite{vg:fgtrees}]\label{t:structure}
Let $T$ be a real $G$-tree. Suppose that $G$ is generic and that $T$
is very small and super stable.  Then, $T$ has a model.

Moreover, there is a model for $T$ that is a graph of spaces
such that each edge space is a point with non-trivial abelian stabilizer and 
each vertex space with restricted lamination is either 
\begin{itemize}
\item\textnormal{(point)} 
a point with non-trivial stabilizer,
\item\textnormal{(linear)}
a non-faithful action of an abelian group on the (2-skeleton of the)
universal cover of a torus with an irrational \footnote{no essential
loops in leaves} lamination, or
\item\textnormal{(surface)} a faithful action of a free group on the universal
cover of a surface with non-empty boundary (represented by points with
$\Z$-stabilizer) with a lamination of pure surface type.
\end{itemize}
\end{thm}

\begin{remark}
For an edge space $\{point\}$, the restriction of the lamination to
$\{point\}\times I$ may or may not be empty. It can be checked that
between any two points in models as in Theorem~\ref{t:structure} there
are $\Lambda$-length minimizing paths. Thin pieces do not arise
because we are assuming our group is freely indecomposable.
\end{remark}

\begin{remark}\label{r:more structure}
Theorem~\ref{t:structure} holds more generally if the assumption that
$G$ is freely indecomposable is replaced by the assumption that $G$ is
freely indecomposable rel point stabilizers, i.e.\ if $\mathcal V$ is
the subset of $G$ of elements acting elliptically\footnote{fixing a
point} on $T$, then $G$ cannot be expressed non-trivially as $A*B$ with
all $g\in\mathcal V$ conjugate into $A\cup B$.
\end{remark}

\begin{figure}
\includegraphics{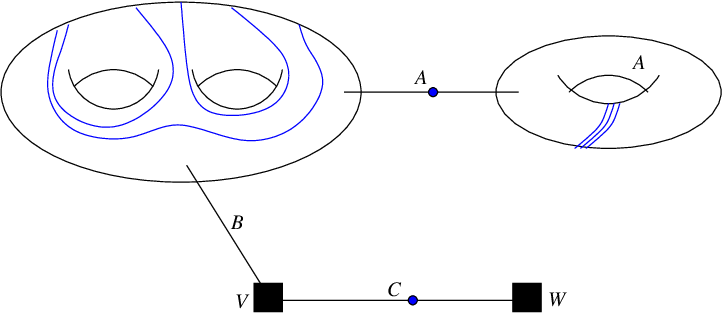}
\caption{A model with a surface vertex space, a linear vertex space,
  and 2 rigid vertex spaces (the black boxes). The groups $A$, $B$ and
  $C$ are abelian with $A$ and $B$ infinite cyclic. Pieces of some
  leaves are also indicated by wavy lines and dots. For example, the
  dot on the edge labeled $C$ is one leaf in a Cantor set of
  leaves.}\label{f:model}
\end{figure}

We can summarize Theorem~\ref{t:structure} by saying that $T$ is a
non-trivial finite graph of simplicial trees, linear trees, and trees
of pure surface type (over trivial trees). See Figure~\ref{f:model}.

\begin{cor}\label{c:structure}
If $G$ and $T$ satisfy the hypotheses of
Theorem~\ref{t:structure}, then $G$ admits a non-trivial \gad\
$\Delta$. Specifically, $\Delta$ may be taken to be the \gad\ induced
by the boundary components of the surface vertex spaces and the
simplicial edges of the model. The surface vertex spaces give rise to
the \qh-vertices of $\Delta$ and the linear vertex spaces give rise to
the abelian vertices of $\Delta$.
\end{cor}

\subsection{Spaces of trees}
Let $G$ be a \fg\ group and let $\A(G)$ be the set of
minimal, non-trivial, real $G$-trees endowed with the Gromov
topology\footnote{The second author thanks Gilbert Levitt for a
helpful discussion on the Gromov and length topologies.}. Recall, see
\cite{fp:trees,fp:gromov=length,bs:paulin}, that in the Gromov
topology $\lim\{(T_n,d_n)\}=(T,d)$ if and only if: for any finite
subset $K$ of $T$, any $\epsilon>0$, and any finite subset $P$ of $G$,
for sufficiently large $n$, there are subsets $K_n$ of $T_n$ and
bijections $f_n:K_n\to K$ such that
$$|d(gf_n(s_n),f_n(t_n))-d_n(gs_n,t_n)|<\epsilon$$
for all $s_n,t_n\in K_n$ and all $g\in P$. Intuitively, larger and larger
pieces of the limit tree with their restricted actions appear in
nearby trees.

Let $\PA(G)$ be the set of non-trivial real $G$-trees modulo
homothety, i.e.\ $(T,d)\sim (T,\lambda d)$ for $\lambda>0$. Fix a
basis for $\f$ and let $T_\f$ be the corresponding Cayley graph. Give
$T_\f$ the simplicial metric. So, a non-trivial homomorphism
$f:G\to\f$ determines $T_f\in\PA(G)$. Let $X$ be the subset of
$Hom(G,\f)$ consisting of those homomorphisms with non-cyclic
image. The space of interest is the closure $\T(G)$ of (the image of)
$\{T_\homo\mid \homo\in X \}$ in $\PA(G)$.

\begin{prop}[\cite{zs:tarski1}]\label{p:basic}
Every sequence of ho\-mo\-morph\-isms in $X$ has a subsequence
$\{f_{n}\}$ such that $\lim T_{f_{n}}=T$ in $\T(G)$. Further,
\begin{enumerate}
\item
$T$ is irreducible\footnote{$T$ is not a line and doesn't have a fixed end}.
\item
$\stabker \homo_n$ is precisely the kernel
$Ker(T)$ of the action of $G$ on $T$.
\item 
The action of $G/Ker(T)$ on $T$ is very small and super stable.
\item\label{i:clopen}
For $g\in G$, $U(g):=\{T\in\T(G)\mid g\in Ker(T)\}$ is
clopen\footnote{both open and closed}.
\end{enumerate}
\end{prop}

\begin{proof}
The initial statement follows from Paulin's
Convergence Theorem \cite{fp:trees}.\footnote{Paulin's proof assumes
the existence of convex hulls and so does not apply in the generality
stated in his theorem. His proof does however apply in our situation
since convex hulls do exist in simplicial trees.} The further items
are exercises in Gromov convergence.
\end{proof}

\begin{warning*}
Sela goes on to claim that stabilizers of minimal components of the
limit tree are trivial (see Lemma~1.6 of \cite{zs:tarski1}). However,
it is possible to construct limit actions on the amalgam of a rank 2
free group $F_2$ and $\Z^3$ over $\Z$ where one of the generators of
$\Z^3$ is glued to the commutator $c$ of basis elements of $F_2$ and
where the $\Z^3$ acts non-simplicially on a linear subtree with $c$
acting trivially on the subtree but not in the kernel of the
action. As a result, some of his arguments, though easily completed,
are not fully complete.
\end{warning*}

\begin{remark}
There is another common topology on $\A(G)$, the length topology. For
$T\in\A(G)$ and $g\in G$, let $\| g\|_T$ denote the minimum distance
that $g$ translates a point of $T$. The length topology is induced by
the map $\A(G)\to [0,\infty)^G$, $T\mapsto (\| g\|_T)_{g\in G}$. Since
the trees in $\A(G)$ are non-trivial, it follows from \cite[page
64]{se:trees} that $\{0\}$ is not in the image\footnote{\label{f:serre} In fact, if
$\mathcal B_G$ is a finite generating set for $G$ then $\|g\|_T\not=0$
for some word $g$ that is a product of at most two elements of
$\mathcal B_G$.}. Since $\T(G)$ consists of irreducible trees, it
follows from \cite{fp:gromov=length} that the two topologies agree
when restricted to $\T(G)$ and from \cite{cm:trees} that $\T(G)$
injects into $\big([0,\infty)^G\setminus\{0\}\big)/(0,\infty)$.
\end{remark}

\begin{cor}\label{c:basic}
$\T(G)$ is metrizable and compact.
\end{cor}

\begin{proof}
$[0,\infty)^G\setminus\{0\}\to\big([0,\infty)^G\setminus
\{0\}\big)/(0,\infty)$ has a section over $\T(G)$ (e.g.\ referring to
Footnote~\ref{f:serre}, normalize so that the sum of the translation lengths of
words in $\mathcal B_G$ of length at most two is one). Therefore,
$\T(G)$ embeds in the metrizable space $[0,\infty)^G$. In light of
this, the main statement of Proposition~\ref{p:basic} implies that $\T(G)$ is compact.
\end{proof}

\begin{remark}
Culler and Morgan \cite{cm:trees} show that, if $G$ is \fg, then $\PA(G)$ with the length topology is compact. This can
be used instead of Paulin's convergence theorem to show that $\T(G)$
is compact. The main lemma to prove is that, in the length topology,
the closure in $\PA(G)$ of $\{T_\homo\mid \homo\in X\}$ consists of
irreducible trees.
\end{remark}

\section{Proof of the Main Proposition}\label{s:proof}
To warm up, we first prove the Main Proposition under the additional
assumption that $\limitgroup$ has only trivial abelian splittings,
i.e.\ every simplicial $\limitgroup$-tree with abelian edge
stabilizers has a fixed point. This proof is then modified to apply to
the general case.

\begin{prop}\label{p:special case}
Suppose that $\limitgroup$ is a generic limit group and has only
trivial abelian splittings\footnote{By Proposition~\ref{p:basic} and
Corollary~\ref{c:structure}, generic limit groups have non-trivial
abelian splittings. The purpose of this proposition is to illustrate
the method in this simpler (vacuous) setting.}. Then, $\limitgroup$
has a factor set.
\end{prop}

\begin{proof}
Let $T\in\T(\limitgroup)$. By Proposition~\ref{p:basic}, either
$Ker(T)$ is non-trivial or $T$ satisfies the hypotheses of
Theorem~\ref{t:structure}. The latter case doesn't occur or else, by
Corollary~\ref{c:structure}, $\limitgroup/Ker(T)$ admits a non-trivial
abelian splitting. In particular, $Ker(T)$ is non-trivial. Choose
non-trivial $k_T\in Ker(T)$. By Item~\ref{i:clopen} of
Proposition~\ref{p:basic}, $\{U(k_T)\mid T\in\T(\limitgroup)\}$ is an
open cover of $\T(\limitgroup)$. Let
$\{U(k_i)\}$ be a finite subcover.  By
definition, $\{\limitgroup\to
Ab(\limitgroup)\}\cup\{q_i:\limitgroup\to \limitgroup/\llangle
k_i\rrangle\}$ is a factor set.
\end{proof}

The key to the proof of the general case is Sela's notion of a {\it short}
homomorphism, a concept which we now define.

\begin{definition}\label{d:length}
Let $G$ be an \fg\ group. Two elements $\homo$ and $\homo'$ in
$Hom(G,\f)$ are {\it equivalent}, denoted $\homo\sim\homo'$, if there
is $\auto\in Mod(G)$ and an element $c\in\f$ such that
$\homo'=i_c\circ\homo\circ\auto$.\footnote{$i_c$ is conjugation by
$c$} Fix a set $\gen$ of generators for $G$ and by $|\homo|$ denote
$\max_{g\in\gen}|\homo(a)|$ where, for elements of $\f$, $|\cdot|$
indicates word length. We say that $\homo$ is {\it short} if, for all
$\homo'\sim\homo$, $|\homo|\le |\homo'|$.
\end{definition}

Note that if $\homo\in X$ and
$\homo'\sim\homo$, then $\homo'\in X$. Here is
another exercise in Gromov convergence. See
\cite[Claim~5.3]{zs:tarski1} and also \cite{rs:structure1} and
\cite[Theorem~7.4]{mb:handbook}.
\begin{exercise}\label{e:tree shorten}
Suppose that $G$ is generic, $\{\homo_i\}$ is a sequence in $Hom(G,\f)$,
and $\lim T_{\homo_i}=T$ in $\T(G)$. Then, either
\begin{itemize}
\item
$Ker(T)$ is non-trivial, or
\item
eventually $\homo_i$ is not short.
\end{itemize}
\end{exercise}
The idea is that if
the first bullet does not hold, then the $\gad$ of $G$ given by
Corollary~\ref{c:structure} can be used to find elements of $Mod(G)$
that shorten $f_i$ for $i$ large.

Let $Y$ be the subset of $X$ consisting of short homomorphisms and let
$\T'(G)$ be the closure in $\T(G)$ of $\{T_f\mid \homo\in Y
\}.$ By
Corollary~\ref{c:basic}, $\T'(G)$ is compact.

\begin{proof}[Proof of the Main Proposition]
Let $T\in\T'(\limitgroup)$. By Exercise~\ref{e:tree shorten}, $Ker(T)$
is non-trivial. Choose non-trivial $k_T\in Ker(T)$. By
Corollary~\ref{c:basic}, $\{U(k_T)\mid
T\in\T'(\limitgroup)\}$ is an open cover of
$\T'(\limitgroup)$. Let $\{U(k_i)\}$ be a finite
subcover.  By definition, $\{\limitgroup\to Ab(\limitgroup)\}\cup\{q_i:\limitgroup\to \limitgroup/\llangle
k_i\rrangle\}$ is a factor set.
\end{proof}

\def\Q{\mathcal Q}

\begin{remark}
Cornelius Reinfeldt and Richard Weidmann point out that a factor set
for a generic limit group $\limitgroup$ can be found without appealing
to Corollary~\ref{c:basic} as follows. Let $\{\gamma_1,\dots\}$
enumerate the non-trivial elements of $\limitgroup$. Let
$\Q_i:=\{\limitgroup/\llangle
\gamma_1\rrangle,\dots,\limitgroup/\llangle\gamma_i\rrangle \}$. If
$\Q_i$ is not a factor set, then there is $\homo_i\in Y$ that is
injective on $\{\gamma_1,\dots,\gamma_i\}$. By Paulin's convergence
theorem, a subsequence of $\{T_{\homo_i}\}$ converges to a faithful
$\limitgroup$-tree in contradiction to Exercise~\ref{e:tree shorten}.
\end{remark}

\jsj-decompositions will be used to prove Theorem~\ref{t:tfae}, so we
digress.

\section{Review: 
  $\mathsf{JSJ}$-theory} Some familiarity with \JSJ-theory is
assumed. The reader is referred to Rips-Sela \cite{rs:jsj},
Dunwoody-Sageev \cite{ds:jsj}, Fujiwara-Papasoglou \cite{fp:jsj}.  For
any freely indecomposable \fg\ group $G$ consider the class \gad's
with at most one edge such that:
\begin{enumerate}
\item[(\JSJ)] every non-cyclic abelian subgroup $A\subset G$ is
elliptic.
\end{enumerate}

We observe that
\begin{itemize}
\item Any two such \gad's are hyperbolic-hyperbolic\footnote{each
edge group of corresponding trees contains an element not fixing a
point of the other tree} or elliptic-elliptic\footnote{each edge group
of corresponding trees fixes a point of the other tree}(a
hyperbolic-elliptic pair implies that one splitting can be used to
refine the other. Since the hyperbolic edge group is necessarily
cyclic by (\JSJ), this refinement gives a free product decomposition
of $G$).
\item A hyperbolic-hyperbolic pair has both edge groups cyclic and
yields a \GAD\ of $G$ with a \QH-vertex group.
\item An elliptic-elliptic pair has a common refinement that satisfies
(\JSJ) and whose set of elliptics is the intersection of
the sets of elliptics in the given splittings.
\end{itemize}

Given a \GAD\ $\Delta$ of $G$, we say that $g\in G$ is {\it
  $\Delta$-elliptic}\footnote{We thank Richard Weidmann for pointing
  out an error in a previous version of this definition and also for
  suggesting the correction.} if there is a vertex group $V$ of
$\Delta$ such that either:
\begin{itemize}
\item $V$ is \QH\ and $g$ is conjugate to a multiple of a boundary component;
\item $V$ is abelian and $g$ is conjugate into $\cl(V)$; or
\item $V$ is rigid and $g$ is conjugate into the envelope of $V$.
\end{itemize}

The idea is that $\Delta$ gives rise to a family of
splittings\footnote{not necessarily satisfying (\jsj ).} with at most
one edge that come from edges of the decomposition, from simple closed
curves in \QH-vertex groups, and from subgroups $A'$ of an abelian
vertex $A$ that contain $\cl(A)$ (equivalently $P(A)$) and with
$A/A'\cong\mathbb Z$. For example, a non-peripheral element of $A$ is
hyperbolic in some 1-edge splitting obtained by blowing up the vertex
$A$ to an edge and then collapsing the original edges of $\Delta$. An
element is $\Delta$-elliptic iff it is elliptic with respect to all
these splittings with at most one edge. Conversely, any finite
collection of \gad's with at most one edge and that sat\-is\-fy (\JSJ)
gives rise to a \GAD\ whose set of elliptics is precisely the
intersection of the set of elliptics in the collection.

\begin{definition}
An {\it abelian \JSJ-decomposition of $G$} is a \GAD\ whose elliptic set
is the intersection of elliptics in the family of {\it all} \gad's
with at most one edge and that satisfy (\JSJ).
\end{definition}

\begin{example}
The group $G=F\times \Z$ has no 1-edge \gad's satisfying (\JSJ) so the
abelian \JSJ-decomposition $\Delta$ of $G$ is a single point labeled
$G$. Of course, $G$ does have (many) abelian splittings. If $F$ is
non-abelian, then every element of $G$ is $\Delta$-elliptic. If $F$ is
abelian, then only the torsion elements of $G$ are $\Delta$-elliptic.
\end{example}

To show that a group $G$ admits an abelian \jsj-decomposition it is necessary
to show that there is a bound to the complexity of the \gad's arising
from finite collections of 1-edge splittings satisfying
(\jsj). If $G$ were \fp\ the results of \cite{bf:bounding}
would suffice. Since we don't know yet that limit groups are \fp,
another technique is needed. Following Sela, we use acylindrical
accessibility.

\begin{definition}
A simplicial $G$-tree $T$ is {\it $n$-acylindrical} if, for
non-trivial $g\in G$, the diameter in the simplicial metric of the
sets $Fix(g)$ is bounded by $n$. It is {\it acylindrical} if it is
$n$-acylindrical for some $n$.
\end{definition}

\begin{thm}[Acylindrical Accessibility:
    Sela \cite{zs:accessibility}, Weidmann
\cite{rw:accessibility}]\label{t:acylindrical accessibility} Let $G$
be a non-cyclic freely indecomposable \fg\ group and let $T$ be a
minimal $k$-a\-cyl\-in\-dri\-cal simplicial $G$-tree. Then, $T/G$ has at most
$1+2k(rank\ G-1)$ vertices.
\end{thm}

The explicit bound in Theorem~\ref{t:acylindrical accessibility} is
due to Richard Weidmann.
For limit groups, 1-edge splittings satisfying (\jsj) are
2-acylindrical and finitely many such splittings give rise to \gad's that
can be arranged to be 2-acylindrical. Theorem~\ref{t:acylindrical
accessibility} can then be applied to show that abelian
\jsj-decompositions exist.

\begin{thm}[\cite{zs:tarski1}]
Limit groups admit abelian \JSJ-decompositions.
\end{thm}

\begin{exercise}[cf.\ Exercises~\ref{e:generators} and
    \ref{e:restriction to M}] If $\limitgroup$ is a generic limit
group, then $Mod(\limitgroup)$ is generated by inner automorphisms
together with generalized Dehn twists associated to 1-edge splittings
of $\limitgroup$ that satisfy (\jsj); see
\cite[Lemma~2.1]{zs:tarski1}. In fact, the only generalized Dehn
twists that are not Dehn twists can be taken to be with respect to a
splitting of the form $A*_C B$ where $A=C\oplus\mathbb Z$.
\end{exercise}

\begin{remark}\label{r:trivial}
Suppose that $\Delta$ is an abelian \jsj-decomposition for a limit
group $G$. If $\blackbox$ is a rigid vertex group of $\Delta$ or the
peripheral subgroup of an abelian vertex of $\Delta$ and if $\auto\in
Mod(G)$, then $\alpha|\blackbox$ is trivial\footnote{Recall our
convention that {\it trivial} means {\it agrees with the restriction
of an inner automorphism}.}. Indeed, $\blackbox$ is $\Delta'$-elliptic
in any 1-edge \gad\ $\Delta'$ of $G$ satisfying (\jsj) and so the
statement is true for a generating set of $Mod(G)$.
\end{remark}

\section{Limit groups are $\mathsf{CLG}$'s}\label{s:tfae}
In this section, we show that limit groups are \clg's and complete the
proof of Theorem~\ref{t:tfae}.

\begin{lemma}\label{l:clg}
Limit groups are \clg's
\end{lemma}

\begin{proof}
Let $\limitgroup$ be a limit group, which we may assume is
generic. Let $\{\homo_i\}$ be a sequence in $Hom(\limitgroup,\f)$ such
that $\homo_i$ is injective on elements of length at most $i$ (with
respect to some finite generating set for $\limitgroup$). Define
$\hat\homo_i$ to be a short map equivalent to $\homo_i$.  According
to Exercise~\ref{e:tree shorten}, $q:\limitgroup\to
\limitgroup':=\limitgroup/\stabker\hat\homo_i$ is a proper
epimorphism, and so by induction we may assume that $\limitgroup'$ is
a \clg.

Let $\Delta$ be an abelian \jsj-decomposition of $\limitgroup$. We
will show that $q$ and $\Delta$ satisfy the conditions in
Definition~\ref{d:clg}. The key observations are these.
\begin{itemize}
\item
Elements of $Mod(\limitgroup)$ when restricted to the peripheral
subgroup $\cl(A)$ of an abelian vertex $A$ of $\Delta$ are trivial
(Remark~\ref{r:trivial}). Since
$\stabker\homo_i$ is trivial, $q|\cl(A)$ is injective.
Similarly, the restriction of $q$ to the envelope of
a rigid vertex group of $\Delta$ is injective.
\item
Elements of $Mod(\limitgroup)$ when restricted to edge groups of
$\Delta$ are trivial.
Since $\limitgroup$ is a limit group, each edge
group is a maximal abelian subgroup in at least one of the two
adjacent vertex groups. See Exercise~\ref{e:max abelian}.
\item
The $q$-image of a \QH-vertex group $Q$ of $\Delta$ is
non-abelian. Indeed, suppose that $Q$ is a \QH-vertex group of
$\Delta$ and that $q(Q)$ is abelian. Then, eventually $\hat\homo_i(Q)$
is abelian. \qh-vertex groups of abelian \jsj-decompositions are
canonical, and so every element of $Mod(\limitgroup)$ preserves $Q$ up
to conjugacy. Hence, eventually $\homo_i(Q)$ is abelian, contradicting the
triviality of $\stabker\homo_i$.
\end{itemize}
\end{proof}

\begin{proof}[Proof of Theorem~\ref{t:tfae}]
(1)$\implies$(2)$\implies$(3) were exercises. (3)$\implies$(1) is the
   content of Lemma~\ref{l:clg}.
\end{proof}

\section{A more geometric approach}\label{s:more geometric}
In this section, we show how to derive the Main Proposition using Rips
theory for \fp\ groups in place of the structure theory of actions of
\fg\ groups on real trees.

\begin{definition}
Let $K$ be a finite 2-complex with a measured lamination
$(\Lambda,\mu)$. The {\it length of $\Lambda$}, denoted $\|\Lambda\|$,
is the sum $\Sigma_e\int_e \mu$ over the edges $e$ of $K$.
\end{definition}

If $\res:\tilde K\to T$ is a resolution, then $\|\res\|_K$ is the
length of the induced lamination $\Lambda_\res$. Suppose that $K$ is a
2-complex for $G$.\footnote{i.e.\ the fundamental group of $K$ is
  identified with the group $G$} Recall that $\Cayley$
is a Cayley graph for $\f$ with respect to a fixed basis and that from
a homomorphism $\homo: G\to\f$ a resolution $\res:(\tilde K,\tilde
K^{(0)})\to (\Cayley,\Cayley^{(0)})$ can be constructed, see
\cite{bf:bounding}. The resolution $\res$ depends on a choice of
images of a set of orbit representatives of vertices in $\tilde K$. If
$\res$ minimizes $\|\cdot\|_K$ over this set of
choices, then we define $\|\homo\|_K:=\|\res\|_K$.

\begin{lemma}\label{l:compare}
Let $K_1$ and $K_2$ be finite 2-complexes for $G$. There is
a number $B=B(K_1,K_2)$ such that, for all $\homo\in Hom(G,\f)$,
$$B^{-1}\cdot \|\homo\|_{K_1}\le\|\homo\|_{K_2}\le B\cdot\|\homo\|_{K_1}.$$
\end{lemma}

\begin{proof}
Let $\res_1:\tilde K_1\to\Cayley$ be a resolution such that
$\|\res_1\|_{K_1}=\|\homo\|_{K_1}.$ Choose an equivariant map
$\psi^{(0)}:\tilde K_2^{(0)}\to\tilde K_1^{(0)}$ between
0-skeleta. Then, $\res_1\psi^{(0)}$ determines a resolution
$\res_2:\tilde K_2\to\Cayley$. Extend $\psi^{(0)}$ to a cellular map
$\psi^{(1)}:\tilde K_2^{(1)}\to\tilde K_1^{(1)}$ between
1-skeleta. Let $B_2$ be the maximum over the edges $e$ of the
simplicial length of the path $\psi^{(1)}(e)$ and let $E_2$ be the
number of edges in $K_2$. Then,
$$\|\homo\|_{K_2}\le\|\res_2\|_{K_2}\le
B_2N_2\|\res_1\|_{K_1}=B_2N_2\|\homo\|_{K_1}.$$ The other inequality
is similar.
\end{proof}

Recall that in Definition~\ref{d:length}, we defined another length
$|\cdot|$ for elements of $Hom(G,\f)$.

\begin{cor}\label{c:length}
Let $K$ be a finite 2-complex for $G$. Then, there is a number
$B=B(K)$ such that for all $\homo\in Hom(G,\f)$ $$B^{-1}\cdot
|\homo|\le \|\homo\|_K\le B\cdot |\homo|.$$
\end{cor}

\begin{proof}
If $\gen$ is the fixed finite generating set for $G$ and if
$R_\gen$ is the wedge of circles with fundamental group identified
with the free group on $\gen$, then complete $R_\gen$ to a 2-complex for $G$
by adding finitely many 2-cells and apply Lemma~\ref{l:compare}.
\end{proof}

\begin{remark}
Lemma~\ref{l:compare} and its corollary allow us to be somewhat
cavalier with our choices of generating sets and 2-complexes.
\end{remark}

\begin{exercise}
The space of (nonempty) measured laminations on $K$ can be identified
with the closed cone without 0 in $\R_+^E$, where $E$ is the set of
edges of $K$, given by the triangle inequalities for each triangle of
$K$. The projectivized space $\pml(K)$
is compact.
\end{exercise}

\begin{definition}
Two sequences
$\{m_i\}$ and $\{n_i\}$ in $\N$ are {\it comparable} if there is a number $C>0$
such that $C^{-1}\cdot m_i\le n_i\le C\cdot m_i$ for all $i$.
\end{definition}

\begin{exercise}\label{e:resolution}
Suppose $K$ is a finite 2-complex for $G$, $\{\homo_i\}$ is a
sequence in $Hom(G,\F)$, $\res_i:\tilde K\to T_\F$ is an
$\homo_i$-equivariant resolution, $\lim T_{\homo_i}=T$, and
$\lim\Lambda_{\res_i}=\Lambda$. If $\{|\homo_i|\}$ and
$\{\|\res_i\|_K\}$ are comparable, then, there is a resolution $\tilde
K\to T$ that sends lifts of leaves of $\Lambda$ to points of $T$ and
is monotonic (Cantor function) on edges of $\tilde K$.
\end{exercise}

\begin{definition}
An element $\homo$ of $Hom(G,\f)$ is {\it $K$-short} if
$\|\homo\|_K\le\|\homo'\|_K$ for all $\homo'\sim\homo$. 
\end{definition}

\begin{cor}
Let $\{\homo_i\}$ be a sequence in $Hom(G,\f)$. Suppose that
$\homo'_i\sim \homo_i\sim \homo_i''$ where $\homo'_i$ is short and
$\homo''_i$ is $K$-short. Then, the sequences $\{|\homo_i'|\}$ and
$\{\|\homo_i''\|_K\}$ are comparable.\qed
\end{cor}

\begin{definition}
If $\leaf$ is a leaf of a measured lamination $\Lambda$ on a finite
2-complex $K$, then (conjugacy classes of) elements in the image of
$\pi_1(\leaf\subset K)$ are {\it carried by $\leaf$}. Suppose that
$\Lambda_i$ is a component of $\Lambda$. If $\Lambda_i$ is simplicial
(consists of a parallel family of compact leaves $\leaf$), then
elements in the image of $\pi_1(\leaf\subset K)$ are {\it carried by
$\Lambda_i$.} If $\Lambda_i$ is minimal and if $N$ is a standard
neighborhood\footnote{see Theorem~\ref{t:rips}} of $\Lambda_i$, then
elements in the image of $\pi_1(N\subset K)$ are {\it carried by
$\Lambda_i$}.
\end{definition}

\begin{definition}\label{d:short sequence}
Let $K$ be a finite 2-complex for $G$. Let $\{\homo_i\}$ be a sequence
of short elements in $Hom(G,\f)$ and let $\res_i:\tilde K\to\Cayley$
be an $\homo_i$-equivariant resolution. We say that the sequence
$\{\res_i\}$ is {\it short} if $\{\|\res_i\|_K\}$ and $\{|\homo_i|\}$
are comparable.
\end{definition}

\begin{exercise}\label{e:shorten}
Let $G$ be freely indecomposable. In the setting of
Definition~\ref{d:short sequence}, if $\{\res_i\}$ is short,
$\Lambda=\lim\Lambda_{\res_i}$, and $T=\lim T_{\homo_i}$, then
$\Lambda$ has a leaf carrying non-trivial elements of $Ker(T)$.
\end{exercise}

The idea is again that, if not, the induced \gad\ could be used
to shorten. The next exercise, along the lines of
Exercise~\ref{e:clg}, will be needed in the following
lemma.\footnote{It is a consequence of Theorem~\ref{t:tfae}, but since
  we are giving an alternate proof we cannot use this.}
\begin{exercise}\label{e:limit group}
Let $\Delta$ be a 1-edge \gad\ of a group $G$ with a
homomorphism $q$ to a limit group $\limitgroup$. Suppose:
\begin{itemize}
\item
the vertex groups of $\Delta$ are non-abelian,
\item
the edge group of $\Delta$ is maximal abelian in each vertex
group, and
\item
$q$ is injective on vertex groups of $\Delta$.
\end{itemize}
Then, $G$ is a limit group.
\end{exercise}

\begin{lemma}\label{l:mod descends}
Let $\limitgroup$ be a limit group and let $q:G\to\limitgroup$ be an
epimorphism such that $Hom(G,\f)=Hom(\limitgroup,\f)$. If $\auto\in
Mod(G)$ then $\auto$ induces an automorphism $\auto'$ of
$\limitgroup$ and $\auto'$ is in
$Mod(\limitgroup)$.
\end{lemma}

\begin{proof}
Since $\limitgroup=RF(G)$, automorphisms of $G$ induce automorphisms
of $\limitgroup$. Let $\Delta$ be a 1-edge splitting of $G$ such that
$\auto\in Mod(\Delta)$. It is enough to check the lemma for
$\auto$. We will check the case that $\Delta=A *_{C} B$ and that
$\auto$ is a Dehn twist by an element $c\in C$ and leave the other
(similar) cases as exercises. We may assume that $q(A)$ and $q(B)$ are
non-abelian for otherwise $\auto'$ is trivial. Our goal is to
successively modify $q$ until it satisfies the conditions of
Exercise~\ref{e:limit group}.

First replace all edge and vertex groups by their $q$-images so that
the third condition of the exercise holds. Always rename the result
$G$. If the second condition does not hold, pull\footnote{If $A_0$ is a
subgroup of $A$, then the result of {\it pulling} $A_0$ across the
edge is $A*_{\langle A_0,C\rangle}\langle A_0,B\rangle$, cf.\ moves of
type IIA in \cite{bf:bounding}.} the centralizers $Z_{A}(c)$ and
$Z_{B}(c)$ across the edge. Iterate. It is not hard to show that the
limiting \gad\ satisfies the conditions of the exercise. So, the
modified $G$ is a limit group. Since $Hom(G,\f)=Hom(\limitgroup,\f)$,
we have that $G=\limitgroup$ and $\auto=\auto'$.
\end{proof}

\begin{proof}[Alternate proof of the Main Proposition]
Suppose that $\limitgroup$ is a generic limit group,
$T\in\T'(\limitgroup)$, and $\{\homo_i\}$ is a sequence of short
elements of $Hom(\limitgroup,\f)$ such that $\lim T_{\homo_i}=T$. As
before, our goal is to show that $Ker(T)$ is non-trivial, so suppose
it is trivial. Recall that the action of $\limitgroup$ on $T$ satisfies all the
conclusions of Proposition~\ref{p:basic}.

Let $q:G\to\limitgroup$ be an epimorphism such that $G$ is \fp\ and
$Hom(G,\f)=Hom(\limitgroup,\f)$. By Lemma~\ref{l:mod descends},
elements of the sequence $\{\homo_iq\}$ are short. We may assume
that all intermediate quotients $G\to G'\to\limitgroup$ are freely
indecomposable\footnote{see \cite{ps:coherence}}.

Choose a 2-complex $K$ for $G$ and a subsequence so that
$\Lambda=\lim\Lambda_{\res_i}$ exists where $\res_i:\tilde K\to T_\f$ is an
$\homo_iq$-equivariant resolution and $\{\res_i\}$ is short. For each component
$\Lambda_0$ of $\Lambda$, perform one of the following moves to obtain
a new finite laminated 2-complex for an \fp\ quotient of $G$ (that we
will immediately rename $(K,\Lambda)$ and $G$). Let $G_0$ denote the subgroup of
$G$ carried by $\Lambda_0$.
\begin{enumerate}
\item
If $\Lambda_0$ is minimal and if $G_0$ stabilizes a linear subtree of
$T$, then enlarge $N(\Lambda_0)$ to a model for the action of $q(G_0)$
on $T$.
\item
If $\Lambda_0$ is minimal and if $G_0$ does not
stabilize a linear subtree of $T$, then collapse all added annuli to
their bases.
\item
If $\Lambda_0$ is simplicial and $G_0$ stabilizes an arc of $T$, then attach
2-cells to leaves to replace $G_0$ by $q(G_0)$.
\end{enumerate}
In each case, also modify the resolutions to obtain a short sequence
on the new complex with induced laminations converging to $\Lambda$.
The modified complex and resolutions contradict
Exercise~\ref{e:shorten}. Hence, $Ker(T)$ is non-trivial.

To finish, choose non-trivial $k_T\in Ker(T)$. As before, if
$\{U(k_{T_i})\}$ is a finite cover for
$\T'(\limitgroup)$, then
$\{\limitgroup\to Ab(\limitgroup)\}\cup\{\limitgroup\to
\limitgroup/\llangle k_{T_i}\rrangle\}$ is a factor set.
\end{proof}

\bibliographystyle{plain} \bibliography{ref}

\begin{thebibliography}{10}

\bibitem{ea:thesis}
Emina Alibegovi{\'c}.
\newblock Makanin-{R}azborov diagrams for limit groups.
\newblock {\em Geom. Topol.}, 11:643--666, 2007.

\bibitem{mb:handbook}
Mladen Bestvina.
\newblock {$\mathbb R$}-trees in topology, geometry, and group theory.
\newblock In {\em Handbook of geometric topology}, pages 55--91. North-Holland,
  Amsterdam, 2002.

\bibitem{bf:bounding}
Mladen Bestvina and Mark Feighn.
\newblock Bounding the complexity of simplicial group actions on trees.
\newblock {\em Invent. Math.}, 103(3):449--469, 1991.

\bibitem{bs:paulin}
M.~R. Bridson and G.~A. Swarup.
\newblock On {H}ausdorff-{G}romov convergence and a theorem of {P}aulin.
\newblock {\em Enseign. Math. (2)}, 40(3-4):267--289, 1994.

\bibitem{cg:limit}
Christophe Champetier and Vincent Guirardel.
\newblock Limit groups as limits of free groups.
\newblock {\em Israel J. Math.}, 146:1--75, 2005.

\bibitem{cm:trees}
Marc Culler and John~W. Morgan.
\newblock Group actions on ${{\bf {R}}}$-trees.
\newblock {\em Proc. London Math. Soc. (3)}, 55(3):571--604, 1987.

\bibitem{ds:jsj}
M.~J. Dunwoody and M.~E. Sageev.
\newblock J{S}{J}-splittings for finitely presented groups over slender groups.
\newblock {\em Invent. Math.}, 135(1):25--44, 1999.

\bibitem{fp:jsj}
K.~Fujiwara and P.~Papasoglu.
\newblock J{SJ}-decompositions of finitely presented groups and complexes of
  groups.
\newblock {\em Geom. Funct. Anal.}, 16(1):70--125, 2006.

\bibitem{gk:surface}
R.~I. Grigorchuk and P.~F. Kurchanov.
\newblock On quadratic equations in free groups.
\newblock In {\em Proceedings of the International Conference on Algebra, Part
  1 (Novosibirsk, 1989)}, volume 131 of {\em Contemp. Math.}, pages 159--171,
  Providence, RI, 1992. Amer. Math. Soc.

\bibitem{dg:mr1}
Daniel Groves.
\newblock Limit groups for relatively hyperbolic group, {I}: the basic tools.
\newblock preprint.

\bibitem{dg:mr2}
Daniel Groves.
\newblock Limit groups for relatively hyperbolic groups. {II}.
  {M}akanin-{R}azborov diagrams.
\newblock {\em Geom. Topol.}, 9:2319--2358 (electronic), 2005.

\bibitem{vg:fgtrees}
Vincent Guirardel.
\newblock Actions of finitely generated groups on {${\bf R}$}-tree.
\newblock to appear in Annales de l'Institut Fourier.

\bibitem{vg:limitgroups}
Vincent Guirardel.
\newblock Limit groups and groups acting freely on {$\mathbb R\sp n$}-trees.
\newblock {\em Geom. Topol.}, 8:1427--1470 (electronic), 2004.

\bibitem{km:limit1}
O.~Kharlampovich and A.~Myasnikov.
\newblock Irreducible affine varieties over a free group. {I}. {I}rreducibility
  of quadratic equations and {N}ullstellensatz.
\newblock {\em J. Algebra}, 200(2):472--516, 1998.

\bibitem{km:limit2}
O.~Kharlampovich and A.~Myasnikov.
\newblock Irreducible affine varieties over a free group. {II}. {S}ystems in
  triangular quasi-quadratic form and description of residually free groups.
\newblock {\em J. Algebra}, 200(2):517--570, 1998.

\bibitem{km:limitgroups}
O.~Kharlampovich and A.~Myasnikov.
\newblock Description of fully residually free groups and irreducible affine
  varieties over a free group.
\newblock In {\em Summer School in Group Theory in Banff, 1996}, volume~17 of
  {\em CRM Proc. Lecture Notes}, pages 71--80. Amer. Math. Soc., Providence,
  RI, 1999.

\bibitem{gl:thin}
Gilbert Levitt.
\newblock La dynamique des pseudogroupes de rotations.
\newblock {\em Invent. Math.}, 113(3):633--670, 1993.

\bibitem{rl:3p}
R.~C. Lyndon.
\newblock The equation {$a\sp{2}b\sp{2}=c\sp{2}$} in free groups.
\newblock {\em Michigan Math. J}, 6:89--95, 1959.

\bibitem{gm:equations}
G.~S. Makanin.
\newblock Equations in a free group.
\newblock {\em Izv. Akad. Nauk SSSR Ser. Mat.}, 46(6):1199--1273, 1344, 1982.

\bibitem{gm:equations2}
G.~S. Makanin.
\newblock Decidability of the universal and positive theories of a free group.
\newblock {\em Izv. Akad. Nauk SSSR Ser. Mat.}, 48(4):735--749, 1984.

\bibitem{ms:valuations1}
J.~Morgan and P.~Shalen.
\newblock Valuations, trees, and degenerations of hyperbolic structures, {I}.
\newblock {\em Ann. of Math.(2)}, 120:401--476, 1984.

\bibitem{fp:trees}
Fr{\'e}d{\'e}ric Paulin.
\newblock Topologie de {G}romov \'equivariante, structures hyperboliques et
  arbres r\'eels.
\newblock {\em Invent. Math.}, 94(1):53--80, 1988.

\bibitem{fp:gromov=length}
Fr{\'e}d{\'e}ric Paulin.
\newblock The {G}romov topology on {${\bf R}$}-trees.
\newblock {\em Topology Appl.}, 32(3):197--221, 1989.

\bibitem{fp:asterisque}
Fr{\'e}d{\'e}ric Paulin.
\newblock Sur la th\'eorie \'el\'ementaire des groupes libres (d'apr\`es
  {S}ela).
\newblock {\em Ast\'erisque}, (294):ix, 363--402, 2004.

\bibitem{ar:equations}
A.~A. Razborov.
\newblock Systems of equations in a free group.
\newblock {\em Izv. Akad. Nauk SSSR Ser. Mat.}, 48(4):779--832, 1984.

\bibitem{vr:2resfree}
V.~N. Remeslennikov.
\newblock {$\exists$}-free groups.
\newblock {\em Sibirsk. Mat. Zh.}, 30(6):193--197, 1989.

\bibitem{vr:rntrees}
V.~N. Remeslennikov.
\newblock {$\exists$}-free groups and groups with length function.
\newblock In {\em Second International Conference on Algebra (Barnaul, 1991)},
  volume 184 of {\em Contemp. Math.}, pages 369--376. Amer. Math. Soc.,
  Providence, RI, 1995.

\bibitem{rs:structure1}
E.~Rips and Z.~Sela.
\newblock Structure and rigidity in hyperbolic groups. {I}.
\newblock {\em Geom. Funct. Anal.}, 4(3):337--371, 1994.

\bibitem{rs:jsj}
E.~Rips and Z.~Sela.
\newblock Cyclic splittings of finitely presented groups and the canonical
  {J}{S}{J} decomposition.
\newblock {\em Ann. of Math. (2)}, 146(1):53--109, 1997.

\bibitem{ps:coherence}
G.~P. Scott.
\newblock Finitely generated {$3$}-manifold groups are finitely presented.
\newblock {\em J. London Math. Soc. (2)}, 6:437--440, 1973.

\bibitem{zs:tarski7}
Z.~Sela.
\newblock Diophantine geometry over groups {VII}: The elementary theory of a
  hyperbolic group.
\newblock preprint.

\bibitem{zs:tarski8}
Z.~Sela.
\newblock Diophantine geometry over groups {VIII}: Stability.
\newblock preprint.

\bibitem{zs:accessibility}
Z.~Sela.
\newblock Acylindrical accessibility for groups.
\newblock {\em Invent. Math.}, 129(3):527--565, 1997.

\bibitem{zs:tarski1}
Z.~Sela.
\newblock Diophantine geometry over groups {I}. {M}akanin-{R}azborov diagrams.
\newblock {\em Publ. Math. Inst. Hautes \'Etudes Sci.}, (93):31--105, 2001.

\bibitem{zs:tarski}
Z.~Sela.
\newblock Diophantine geometry over groups and the elementary theory of free
  and hyperbolic groups.
\newblock In {\em Proceedings of the International Congress of Mathematicians,
  Vol. II (Beijing, 2002)}, pages 87--92, Beijing, 2002. Higher Ed. Press.

\bibitem{zs:tarski2}
Z.~Sela.
\newblock Diophantine geometry over groups {II}. {C}ompletions, closures and
  formal solutions.
\newblock {\em Israel J. Math.}, 134:173--254, 2003.

\bibitem{zs:tarski4}
Z.~Sela.
\newblock Diophantine geometry over groups {IV}. {A}n iterative procedure for
  validation of a sentence.
\newblock {\em Israel J. Math.}, 143:1--130, 2004.

\bibitem{zs:tarski3}
Z.~Sela.
\newblock Diophantine geometry over groups {III}. {R}igid and solid solutions.
\newblock {\em Israel J. Math.}, 147:1--73, 2005.

\bibitem{zs:tarski5.1}
Z.~Sela.
\newblock Diophantine geometry over groups {$\rm V\sb 1$}. {Q}uantifier
  elimination. {I}.
\newblock {\em Israel J. Math.}, 150:1--197, 2005.

\bibitem{zs:tarski5.2}
Z.~Sela.
\newblock Diophantine geometry over groups {${\rm V}\sb 2$}. {Q}uantifier
  elimination. {II}.
\newblock {\em Geom. Funct. Anal.}, 16(3):537--706, 2006.

\bibitem{zs:tarski6}
Z.~Sela.
\newblock Diophantine geometry over groups {VI}. {T}he elementary theory of a
  free group.
\newblock {\em Geom. Funct. Anal.}, 16(3):707--730, 2006.

\bibitem{se:trees}
Jean-Pierre Serre.
\newblock {\em Trees}.
\newblock Springer-Verlag, Berlin, 1980.
\newblock Translated from the French by John Stillwell.

\bibitem{js:surface}
John Stallings.
\newblock How not to prove the {P}oincar\'e conjecture.
\newblock In {\em Topology {S}eminar, {W}isconsin, 1965}, Edited by R. H. Bing
  and R. J. Bean. Annals of Mathematics Studies, No. 60, pages ix+246.
  Princeton University Press, 1966.

\bibitem{rw:accessibility}
Richard Weidmann.
\newblock The {N}ielsen method for groups acting on trees.
\newblock {\em Proc. London Math. Soc. (3)}, 85(1):93--118, 2002.

\bibitem{henry}
Henry Wilton.
\newblock Solutions to {B}estvina and {F}eighn's {E}xercises on limit groups.
\newblock math.GR/0604137.

\bibitem{hz:surfaces}
Heiner Zieschang.
\newblock Alternierende {P}rodukte in freien {G}ruppen.
\newblock {\em Abh. Math. Sem. Univ. Hamburg}, 27:13--31, 1964.

\end{thebibliography}
\end{document}